\documentclass[a4paper,10pt]{amsart}
\usepackage[dvips,all,color,line]{xy}
\usepackage{amsmath, amsthm, amssymb}
\usepackage{enumerate,hyperref}
\usepackage{graphicx}

\newcommand{\dom}{\partial_{\bullet}}
\newcommand{\cod}{\partial^{\bullet}}
\newcommand{\ass}[1]{\mathfrak{gr}_{\mathfrak{m}}(#1)}
\newcommand{\ie}{{\it i.e.}}
\newcommand{\cat}{\mathcal{C}}
\newcommand{\cover}{\prec}

\newtheorem{thm}{Theorem}[section]

\newtheorem{theorem}[thm]{Theorem}
\theoremstyle{definition}
\newtheorem{example}[thm]{Example}
\newtheorem{definition}[thm]{Definition}
\newtheorem{remark}[thm]{Remark}

\title[Incidence algebras of shellable posets]{The incidence algebra of posets and acyclic categories}
\author{David Quinn}

\address{David Quinn\\ Queen's University Belfast\\ School of
  Mathematics and Physics\\ Pure Mathematics Research Centre\\ Belfast
  BT7~1NN\\ Northern Ireland\\ UK}
\curraddr{David Quinn\\ University of Aberdeen\\ School of Natural and Computing Sciences\\ 
Institute of Mathematics\\ Fraser Nobel Building\\ Aberdeen AB24~3UE\\ Scotland\\ UK}
\email{davidquinnmath@gmail.com}


\begin{document}
\begin{abstract}
Acyclic categories were introduced by Kozlov and can be viewed as generalised posets. Similar to posets, 
one can define their incidence algebras and a related topological complex. We consider the incidence algebra of 
either a poset or acyclic category as the quotient of a path algebra by the parallel ideal. We show that this ideal 
has a quadratic Gr\"obner basis with a lexicographic monomial order if and only if the poset or acyclic category is 
lex-shellable.
\end{abstract}
\maketitle
\section{Introduction}
It has been shown independently by Polo \cite{polo} and Woodcock \cite{woodcock} that for a graded finite poset 
$P$ there is an equivalence between the incidence algebra being Koszul and the order complex $\Delta P$ being Cohen-Macaulay. 
Recently Reiner and Stamate \cite{reinerstamate} have proved that the graded hypothesis is largely unnecessary by showing a 
similar equivalence between non-graded Koszul and sequentially Cohen-Macaulay. Here we define the incidence algebra as a 
quotient of a quiver with an ideal called the parallel ideal. Our first result shows that, for a not necessarily graded 
finite poset, the parallel ideal has a quadratic Gr\"obner basis for a lexicographic monomial order if and only if $P$ 
is lex-shellable. 
Since shellability implies Cohen-Macaulay an even more striking relationship between the combinatorial and topological properties of the poset and the algebraic 
properties of the incidence algebra.

The most complete account of acyclic categories can be found in \cite{koz} where most of the theory first appears. 
Acyclic categories can be thought of as a generalisation of posets and in analogy we can define incidence algebras 
and generalised order complexes for acyclic categories.  We collect the necessary definitions at the beginning of 
Section \ref{ac}. This is not intended as an introduction to acyclic categories but merely an attempt to be self contained. 
We will show that there is no extra difficulty in extending our result to lex-shellable acyclic categories.

\section{Incidence algebras of lex-shellable posets}
\subsection{Shellable and lex-shellable posets}

We say that a poset $P$ is bounded if there exists $a,\,b\in P$ such that $a\leq x \leq b$ for all
$x\in P$. The elements $a,b$ are refereed to the {\em minimal} and {\em maximal} elements of $P$ respectively. 
For any poset $P$ we construct the {\em augmented} poset $\hat{P}$ by adjoining two new elements $\hat{0},\hat{1}$ to $P$ 
where $\hat{0}$ is the minimal element of $\hat{P}$ and $\hat{1}$ is the maximal element of $\hat{P}$.

We shall work with bounded posets throughout. The definition of a lex-shellable poset $P$ typically gives a
condition on the augmented poset $\hat{P}$, however, since we take $P$ to be bounded we may apply the condition directly to $P$.

Recall that the Hasse diagram of a poset $P$ is a directed graph in which the vertices
are the elements of $P$ and there exists a directed edge $x\rightarrow y$ whenever $y$ 
covers $x$, \ie, there does not exist $z\in P$ such that $x<z<y$. 
We use the notation $x\cover y$ to denote that $y$ covers $x$.

We work with lex-shellable posets however it is important to note the more general notion of shellable posets.

\begin{definition}\label{ordercomplex}
For a poset $P$ we define its {\em order complex} of $\Delta P$ to be the abstract simplicial 
complex with a face $\{x_1,\dots,x_k\}$ for $x_1 < \dots < x_k$ a chain in $P$.
\end{definition}

\begin{definition}\label{def:shell}
A (generalised) simplicial complex $\Delta$ is {\em shellable} if the facets, $F_i$, of $\Delta$ can be linearly 
ordered such that for every pair $(F_i,F_j)$ where $1\leq i < j$ there exists some 
$k \leq j$ such that $F_i \cap F_j \subset F_k \cap F_j = F_j \backslash \{x\}$, $x\in F_j$.
\end{definition}

In order to define lex-shellability we introduce edges labellings. An {\em edge labelling} of $P$ is a map from the set of edges 
of that Hasse diagram of $P$ to some poset $\Lambda$, such a labelling induces a labelling on the paths 
of the Hasse diagram as sequences of elements of $\Lambda$. We assume the {\em prefix condition}, that is,
given any two maximal paths $p_1,\,p_2$ in some closed interval $[x,y]$ of $P$, the labelling of $p_1$ is not a prefix of the labelling
of $p_2$. 

\begin{definition}\label{lexcondition}
An edge labelling of a poset $P$ is a LEX-labelling if it satisfies the LEX-condition, given as follows:
For any interval $[x,y]$, any maximal chain $C$ in $[x,y]$, and any $s,t \in C$ such that 
$x<s<t<y$, if $C|_{[x,t]}$ is lexicographically least in $[x,t]$ and $C|_{[s,y]}$ is 
lexicographically least in $[s,y]$, then $C$ is lexicographically least in $[x,y]$.
\end{definition}

It will be more convenient to use an equivalent condition known as
as the {\em short bad subchain} condition, or SBS-condition, which we define as follows.

\begin{definition}\label{def:sbs}
An edge labelling of a poset $P$ is a LEX-labelling if it satisfies the SBS-condition, given as follows:
For any interval $[x,y]$, any maximal chain $C$ in $[x,y]$, $C$ is not lexicographically 
least in $[x,y]$ then there exists $r,s,t\in C$ with $r\cover s \cover t$ such that
and $C|_{[r,t]}$ is not lexicographically least in $[r,t]$.
\end{definition}

We say that a poset is lex-shellable if there exists a LEX-labelling, that is, 
if there exists an edge labelling of $P$ which satisfies the equivalent conditions
from Definitions \ref{lexcondition} and \ref{def:sbs}, note that this implies shellability as in 
Definition \ref{def:shell}. A proof of the equivalence of these conditions can be found in \cite{koz}.

\subsection{Incidence algebras}
For a path $p$ in the Hasse diagram we define the {\em domain} of $p$, denoted $\dom p$ to be 
the initial vertex (or starting vertex) of the path $p$. Dually we define the {\em codomain}, denoted $\cod p$,
to be the terminal vertex (or destination vertex) of $p$.  
The quiver of a poset $P$ is the $k$ algebra $\Gamma_k[P]$ with $k$ basis given by the paths 
in the Hasse diagram of $P$. Multiplication is given by concatenation of paths when defined 
and is zero otherwise, \ie, if $p,q$ are paths the Hasse diagram of $P$ then 
\[p\circ q=\begin{cases}
      pq & \cod p=\dom q \\
      0  & \mathrm{otherwise}.
     \end{cases}\]
We shall omit $k$ from our notation so that the quiver will be denoted $\Gamma[P]$.

\begin{definition}
The {\em parallel ideal} $I_P$ is the ideal of $\Gamma[P]$ generated by the relations 
$p_1 - p_2$ whenever $\dom p_1 = \dom p_2$ and $\cod p_1 = \cod p_2$. The {\em incidence algebra} 
$k[P]$ is the quotient of $\Gamma[P]$ by the parallel ideal. 
\end{definition}

Let $\text{Int}(P)$ be the set of closed intervals of $P$. The incidence algebra 
$k[P]$ may also be viewed as the $k$-vector space with basis $\{\xi_{[x,y]}\}$ for 
$[x,y] \in {\text{Int}(P)}$ and multiplication defined by
\begin{displaymath}
\xi_{[x,y]}\xi_{[z,w]}=\delta_{yz}\xi_{[x,w]}
\end{displaymath}
where $\delta_{yz}$ is the Kronecker delta. 
Also, we define $\xi_{x}:=\xi_{[x,x]}$ which has degree $0$.

A monomial of the quiver $\Gamma[P]$ 
of degree $n>0$ is given by a path $x_0 \rightarrow \dots \rightarrow x_n$ 
of length $n$. We represent this a little more succinctly as $\xi_{[x_0,x_1]}\dots \xi_{[x_{n-1},x_n]}$
with $x_i \prec x_{i+1}$ a covering relation and each $\xi_{[x_i,x_{i+1}]}$ has degree $1$.
Note that $k[P]$ can be decomposed as $\Gamma_0 \oplus \mathfrak{m}$ where 
$\Gamma_0$ is the $k$ vector space with basis $\{\xi_x \mid x\in P\}$ and $\mathfrak{m}$
is the $k$-span of the intervals $\{ \xi_{[x,y]} \mid x<y\}$.

We say that a poset is {\em graded} if and only if all maximal chains have the same 
length. In that case $I_P$ is a homogeneous ideal and the incidence algebra is 
$\mathbb{N}$-graded with 
$\text{deg}{\xi_{[x,y]}}=\text{length}([x,y])$. We will work with not necessarily graded 
posets and algebras. In order to relate our results to those of Reiner and Stamate \cite{reinerstamate} we must 
consider the associated graded algebra $\ass{k[P]}$ which is defined as follows.
\[\ass{k[P]} = \Gamma_0 \oplus \mathfrak{m} / \mathfrak{m}^2 \oplus \mathfrak{m}^2 / \mathfrak{m}^3 \oplus \dots\]
Starting from our definition of the incidence algebra as the quotient $k[P]=\Gamma[P]/I_P$ the associated graded algebra
$\ass{k[P]}$ will be a $k$-vector space on the same basis as $k[P]$, namely the closed intervals of $P$,
however the algebraic structure will be different. The degree of $\xi_{[x,y]}$ in $\ass{k[P]}$ is the maximal length of the 
maximal chains in $[x,y]$, which we denote $\deg[x,y]$, and the product is given as follows.
\[\xi_{[x,y]} \xi_{[w,z]} = \begin{cases}
                           \xi_{[x,y]} & y=w \text{ and } \deg[x,y]+\deg[w,z]=\deg[x,z] \\
                            0          & \text{otherwise}.  
                          \end{cases}\]
We wish to express $\ass{k[P]}$ as a quotient 
$\ass{\Gamma[P]}/J={\Gamma[P]}/J=$ where $J$ is the kernel of the natural map $\Gamma[P]\rightarrow \ass{k[P]}$.
If a monomial $p$ in $\Gamma[P]$ corresponds to a path which is not maximal within the relevant closed interval then $p\in J$.
The ideal $J$ also contains binomials $q-q'$ where $q$ and $q'$ correspond to maximal paths within the same closed interval. 
Thus $J$ is the ideal generated by $\{p,q-q'\}$.

If $f\in I_P$ is homogeneous, \ie, $f \in \mathfrak{m}^i{k[P]}/\mathfrak{m}^{i+1}{k[P]}$ for some $i>1$ then clearly $f \in J$.

We consider the ideal $J'=\{t(f) \mid f \in I_p\}$ where $\mathrm{t}(f)$ is the {\em truncation} of $f$ defined as follows.
\[t(f)= f\cap \mathfrak{m}^i/\mathfrak{m}^{i+1}\]
where $i$ is such that $f \in \mathfrak{m}^i$ but $f \notin \mathfrak{m}^{i+1}$.
In other words, the truncation of $f$ is the sum of the terms of least degree in $f$.

Certainly $J \subset J'$ as each generator of $J$ is the truncation of some binomial in $I_P$.
 $I_P$ is an element of $J$. Note that any $f\in I_P$ can be expressed as a finite sum of 
binomials $\sum_i p_i-q_i$ where each pair $p_i,\,q_i$ 
correspond to paths sharing the same domain and codomain. We may assume that each $t(p_i-q_i)$ contributes to $t(f)$,
and it follows that $t(f)\in J$ and so $J'=J$. 

The monomial order we employ to find a Gr\"obner basis for $I_P$ will be made explicit in Definition \ref{def:lexorder}. 
This order has the property that monomials of lower degree are ordered first, thus the initial monomials 
$\mathrm{in}(f) = \mathrm{in}(t(f))$ 
and if $\mathcal{G}$ is a Gr\"obner basis for $I_P$ then $t(\mathcal{G})$
will be a Gr\"obner basis for $J$. In particular $\mathrm{in}(f)$ is quadratic for all $f \in \mathcal{G}$ if and only if
$t(\mathcal{G})$ is quadratic in the usual sense. For this reason we make the following definition.

\begin{definition}
Let $\mathcal{G}$ be a Gr\"obner basis for $I_P$ with $P$ not necessarily graded. We say that $\mathcal{G}$
is {\em quadratic} if $t(\mathcal{G})\subset \mathfrak{m}^2/\mathfrak{m}^3$.
\end{definition}

For our Gr\"obner basis we will consider a lexicographic order for which the monomials of 
lower degree come later in the order than those of higher degree. Although this is only necessary when 
$P$ is non-graded, we take this order so that we can treat both cases at once. 

A LEX-labelling gives a total order on the maximal chains of the poset, this is equivalent to a total order on the 
corresponding `maximal' monomials of the incidence algebra. We wish to extend this to a lexicographic monomial
order on the monomials of the incidence algebra. For a monomial $w$ we denote $l(w)$ to be
the LEX-labelling of chain corresponding to $w$. 
If the LEX-labelling is injective then we order the monomials 
of the incidence algebra so that for monomials $w$, $v$ we set $w>v$ if:
\begin{itemize}
 \item $\mathrm{deg}\,w < \mathrm{deg}\,v,$ or
 \item $\mathrm{deg}\,w = \mathrm{deg}\,v$ and $l(w) > l(v)$.
\end{itemize}

In general, a LEX-labelling might not give a total order on the monomials
of the incidence algebra, we break any ties by using the following order.

\begin{definition}\label{def:lexorder}
 Given a lex-shelling of a poset $P$ we construct a well order $<$ on the monomials of $\Gamma[P]$ as follows. 
 We label the maximal chains by the total order given by the lex-shelling so that $C_i$ is the $i$th maximal chain. 
 To each maximal chain $C_i$ there exists a corresponding `maximal' monomial $m_i$. For any monomial $w$ 
 we define the {\em carrier} of $w$ to be earliest chain $C_i$ in the shelling such that $w$ divides $m_i$. 
 We denote the carrier of $w$ by $C_w$.

 For monomials $w$, $v$ we set $w>v$ if:
\begin{itemize}
 \item $\mathrm{deg}\,w < \mathrm{deg}\,v,$ or
 \item $\mathrm{deg}\,w = \mathrm{deg}\,v$ and $l(w) > l(v)$, or
 \item $\mathrm{deg}\,w = \mathrm{deg}\,v$ and $l(w) = l(v)$ and $C_w < C_v$, or
 \item $\mathrm{deg}\,w = \mathrm{deg}\,v$ and $l(w) = l(v)$ and $C_w = C_v$  and $\dom w>\dom v $
\end{itemize}
 where the order on the domains is that of $P$.

 We call a monomial order for which $\deg\, w < \deg\, v$ implies $v<w$ a {\em negative degree monomial order} as
 this condition can be given as $-\!\deg\, v < -\!\deg\, w$ implies $v<w$.
\end{definition}

Let $w$ and $v$ be monomials (of the same degree) assigned the same label by the LEX-labelling. We assume $w<v$ as determined
by the carrier chains $C_w < C_v$. Given some monomial $u$ such that $uw$ and $uv$ are non-zero we must ensure that $uw<uv$.
Note that the domains of $w$ and $v$ must be the same since $uw$ and $uv$ are both non-zero.
The carrier of $w$ consists of the union of the lexicographically least maximal chain in the interval $[\hat{0},\dom w]$, the maximal
chain in $[\dom w, \cod w]$ corresponding to $w$ and the lexicographically least maximal chain in $[\cod w, \hat{1}]$. Similar 
holds for $C_v$ thus the LEX-labelling of $C_w|_{[\hat{0},\cod w]}$ is the same as the LEX-labelling of $C_v|_{[\hat{0},\cod v]}$.
Similarly we can show that the LEX-labelling of $C_{uw}|_{[\hat{0},\cod w]}$ is the same as the LEX-labelling of $C_{uv}|_{[\hat{0},\cod v]}$.
The chains $C_w|_{[\cod w,\hat{1}]}$ and $C_{uw}|_{[\cod w,\hat{1}]}$ agree, and similarly $C_w|_{[\cod w,\hat{1}]} = C_{uw}|_{[\cod w,\hat{1}]}$.
The order $C_w <C_v$ must come from the LEX-labelling of the chains in the previous sentence and it follows that $C_{uw}<C_{uv}$.
With the same hypothesis on $w$ and $v$ and with $wu$ and $vu$ non-zero we can show $wu<vu$ with a similar argument.

\begin{definition}\label{def:neglexorder}
Given a labelling of the Hasse diagram of a poset $P$ which satisfies the prefix condition, a {\em negative degree lexicographic monomial 
order} for $\Gamma[P]$ is any
negative monomial order such that the order on the monomials of equal degree is a refinement of the lexicographic order with 
respect to the labelling.
\end{definition}
 
Note that since the labelling must satisfy the prefix condition no two distinct maximal chains in the same closed interval can
be labeled the same. The well order given in Definition \ref{def:lexorder} is a negative degree lexicographic monomial order. 
If $w$, $u$, $v$, $s$ are monomials in $\Gamma[P]$ then this order satisfies:
\begin{enumerate}
 \item If $w<u$ and $vws\neq0\neq vus$ then $vws<vus$.
 \item If $u=vws\neq 0$ then $u<w$.
\end{enumerate}
Note that the second condition is the reverse of the usual condition used for monomial orders, 
yet the monomial order is a well order as there are only finitely many paths in the Hasse diagram of $P$.

\begin{example}\label{ex:order}
We consider the lex-shellable poset given by the Hasse diagram in Fig.~\ref{fig:gbsh}, where we also give an (injective) 
$\mathbb{N}$-labelling of the edges. The Gr\"obner basis for the parallel ideal is as follows.
\begin{displaymath}
 \mathcal{G}=\{
\xi_{[a,d]}\xi_{[d,g]}-\xi_{[a,b]}\xi_{[b,e]}\xi_{[e,g]},\,
\xi_{[b,e]}\xi_{[e,g]}-\xi_{[b,f]}\xi_{[f,g]},\,
\xi_{[a,b]}\xi_{[b,f]}-\xi_{[a,c]}\xi_{[c,f]}
\}
\end{displaymath}
Note that the order we use forces the quadratic term of $\xi_{[a,d]}\xi_{[d,g]}-\xi_{[a,b]}\xi_{[b,e]}\xi_{[e,g]}$
to be its initial term. For the remaining elements of the Gr\"obner basis the order comes from the labelling.
The initial terms are all quadratic so we have a quadratic Gr\"obner basis.
\end{example}
\begin{figure}
\[
\xymatrix{
{}                           &{}&{\mathbf{g}}\ar@{-}[dddll]_{9} \ar@{-}[dd]_{7} \ar@{-}[ddrr]^{3} &{}&{}                             \\
{}                           &{}&{}                                                               &{}&{}                             \\
{}                           &{}&{\mathbf{e}}\ar@{-}[dd]_6                              &{}&{\mathbf{f}}\ar@{-}[dd]^2\ar@{-}[ddll]_5 \\
{\mathbf{d}}\ar@{-}[dddrr]_8 &{}&{}                                                               &{}&{}                             \\
{}                           &{}&{\mathbf{b}}\ar@{-}[dd]_4                                        &{}&{\mathbf{c}}\ar@{-}[ddll]^1    \\
{}                           &{}&{}                                                               &{}&                               \\
{}                           &{}&{\mathbf{a}}                                                     &{}&                         
}
\]
\caption{Lex-shelling}
\label{fig:gbsh}
\end{figure}
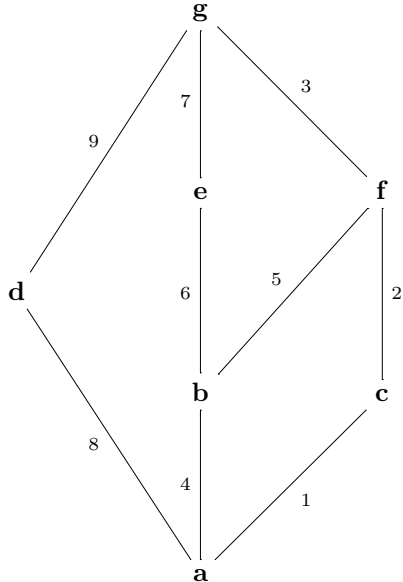

\subsection{The main result}

\begin{theorem}\label{thm:poset}
The parallel ideal $I_P$ of a finite poset $P$ has a quadratic Gr\"obner basis with a negative degree lexicographic monomial 
order if and only if $P$ is lex-shellable.
\end{theorem}

\begin{proof}
($\impliedby$) We assume that $P$ is lex-shellable and fix some lex-shelling. We take a monomial 
order as described in Definition \ref{def:lexorder}.

The initial ideal of $I_P$ will be minimally generated by monomials arising from paths in the directed Hasse 
diagram. Each such monomial will be the initial term of at least one binomial in $I_P$. Let $m$ be a 
monomial which is the initial term of some element of $I_P$, and assume that $m$ is at least cubic. 
Let $F_m$ be the chain corresponding to $m$ 
(equivalently, $F_m$ is the face of $\Delta(P)$ corresponding to $m$). 
We take the closed interval in which $F_m$ is a maximal chain. Since the 
monomial order coincides with the shelling order $F_m$ cannot be lexicographically least in this 
interval. [Note that the shelling is written with the least element appearing first, while the leading monomial is
the greatest with respect to the order.] By the SBS-condition there exists a path $x\cover y \cover z$ such that $F_m|_{[x,z]}$ is not 
lexicographically least in $[x,z]$. The restriction $F_m|_{[x,z]}$ corresponds to a quadratic monomial which, by the lexicographic order,
is the initial term of some element of $I_P$, thus $m$ is divisible by a quadratic initial term and the monomial order
gives a quadratic Gr\"obner basis for $I_P$. 

($\implies$) Now assume negative degree lexicographic monomial order for $\Gamma[P]$ 
which gives a quadratic Gr\"obner basis for $I_P$. This monomial order arises from some 
labelling of the of the Hasse diagram of $P$ which satisfies the prefix condition.
We only compare monomials corresponding to maximal chains from the same closed interval and the 
labelling gives a total order on the sets of maximal chains from the same closed interval. Any choice made 
to give a negative lexicographic monomial order on $\Gamma[P]$ does not influence the proof. 
We will show that this labelling is in fact a LEX-labelling by showing that it satisfies the SBS-condition.  

Denote by $\mathcal{G}_{[x,z]}$ the Gr\"obner basis of the parallel ideal of the interval $[x,z]$ 
with respect to the monomial order for $\Gamma[P]$ restricted to $[x,z]$.
Each $\mathcal{G}_{[x,z]}$ must also be quadratic  
as the ideal $I_{[x,z]}$ is generated by a subset of the generators of $I_P$, and any element of $\mathcal{G}_{[x,z]}$ with 
non-quadratic initial term would also be in $\mathcal{G}$. This is because no quadratic monomial dividing it 
could be the initial term of any polynomial in $I_P$ not occurring in $I_{[x,z]}$.

Let $C$ be a maximal chain in the interval $[x,z]$ 
which is not lexicographically least in $[x,z]$. Thus $C$ corresponds to a monomial $m$ which is the initial term of some element
of $I_{[x,z]}$. Since we have a quadratic Gr\"obner basis, $m$ must be divisible by a quadratic monomial which is the initial term of some
binomial in $I_{[x,z]}$, call this quadratic monomial $q$, and note that $q$ corresponds to a maximal chain $F$ in some interval
$[y,w]\subset[x,z]$. As $q$ is the initial term of some binomial in $I_{[y,w]}$, the maximal chain $F$ is not lexicographically least in
$[y,w]$, thus the chain $C$ satisfies the SBS-condition.
\end{proof}

\begin{remark}
There is potential that a similar result holds in the case of shellable but not lex-shellable posets. Examples of such posets
can be found in \cite{nonshellwachs,walker}. 
A given shelling might not be consistent with a monomial order but it may always be possible to construct 
a shelling which is consistent with some monomial order. For a non lex-shellable poset the monomial order would not be lexicographic.

If we assume the existence of a non-lexicographic monomial order such that the Gr\"obner basis of the parallel ideal is quadratic then
it is possible to show that the poset is shellable with a shelling given by the induced order on the maximal chains. 
It is unknown if there exists such a monomial order, \ie, for $P$ shellable but not lex-shellable does there exist a 
(necessarily non-lexicographic) monomial order for $\Gamma[P]$ which gives a quadratic Gr\"obner basis for the parallel ideal.
\end{remark}

\section{Incidence algebras of lex-shellable acyclic categories}\label{ac}
\subsection{Acyclic categories}
We open this section by introducing acyclic categories. Much of this follows \cite{koz} with only minor notational differences. 
We consider only finite augmented acyclic categories, that is, acyclic categories with finitely many objects 
and morphisms, and also initial and terminal objects. This is similar to our consideration of finite bounded posets.

\begin{definition}
An {\em acyclic category} $\cat $ is a small category in which only the identity morphisms have inverses and any 
morphism from an object to itself is an identity morphism. The category is {\em finite} if the class of objects $\mathcal{O}(\cat)$, 
and all $\mathrm{hom}(x,y)$ are finite sets. The category is {\em augmented} if it has both initial and terminal objects.
\end{definition}

This generalises posets in the sense that a poset is an example of an acyclic category.

\begin{example}
A poset $P$ is an acyclic category whose objects are the elements of $P$ and whose morphism sets $\text{hom}(x,y)$ 
contain precisely one element if and only if $x\leq y$ in $P$, and are empty otherwise.
\end{example}

The notion of an order complex generalises also (often referred to as the nerve of the category), in this case we 
do not have an abstract simplicial complex unless $\cat$ is a poset. Following Kozlov \cite{koz} we refer to these 
complexes as generalised simplicial complexes. A generalised simplicial complex consists of simplices but unlike a simplicial
complex each $k$-face may not determined by its $(k-1)$-faces. A simple example is the generalised simplicial complex with two 
$0$-faces and two $1$-faces having the same boundary.

\begin{definition}\label{deltac}
 The {\em nerve} of an acyclic category is the generalised simplicial complex with vertex set given by the objects of $\cat$ (more properly
the $0$-faces are the identity morphisms for each object). The $1$-faces are the morphisms $m_1$, the boundary 
of $m_1$ is $\{\mathrm{id}_{\dom m_1},\mathrm{id}_{\cod m_1}\}$. 
For $k>1$, the $k$-faces are given by chains of 
composable (non-identity) morphisms $m_1 \circ m_2 \circ \dots \circ m_k$. The boundary of each face is given by
\begin{enumerate}
 \item $m_2 \circ \dots \circ m_k$,
 \item $m_1 \circ \dots \circ (m_i \circ m_{i+1}) \circ \dots \circ m_k$,
 \item $m_1 \circ m_2 \circ \dots \circ m_{k-1}$,
\end{enumerate}
where in ($2$) $(m_i \circ m_{i+1})$ represents the single morphism given by the composition, and as such is also 
a chain of length $k-1$.
\end{definition}

For example, the nerve of the acyclic category from Example \ref{exac} and Figure \ref{acfig} has a pair of 
$2$-simplices sharing the edges corresponding to the morphisms $\beta$ and $\beta \circ \alpha_i$ (where this 
is considered as a morphism rather than a chain), in addition it has a maximal $1$-face corresponding the morphism 
$\gamma$. Colloquially, in this example the nerve is a cone with a handle on its side. 
  
The {\em indecomposable morphisms} are the non-identity morphisms which cannot be written as the composition of 
two non-identity morphisms. The indecomposable morphisms in a poset are the covering relations. 

For each finite acyclic category $\cat $ we can construct a directed graph. The vertices are given by the 
objects $\mathcal{O(C)}$ and there exists a directed edge $x\rightarrow y$ for each element of $\text{hom}(x,y)$. 
This is related to the directed Hasse diagram of a poset, however in that case we only had edges for covering 
relations since the morphisms were uniquely determined by their domain and codomain. Similarly, in a poset, 
the maximal chains in a closed interval $[x,y]$ correspond to the decompositions of the unique morphism in 
$\text{hom}(x,y)$ into indecomposable morphisms. For acyclic categories the `intervals' are the morphisms and 
their decompositions into indecomposable morphisms. 

\begin{example}\label{exac}
Figure \ref{acfig} represents the acyclic category with three objects and four indecomposable morphisms 
$\{\alpha_1,\alpha_2,\beta,\gamma\}$ where $\beta \circ \alpha_1=\beta \circ \alpha_2 \neq \gamma$. Note 
that this is not an augmented acyclic category, and also we have omitted edges for the identity morphisms.
\end{example}

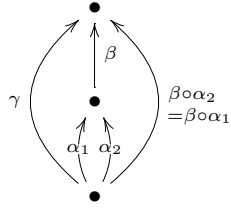
\begin{figure}
 \[
 \xymatrix{
\bullet \\
\bullet \ar[u]_\beta \\
\bullet \ar @/_/[u]|{\alpha_2} \ar @/^/[u]|{\alpha_1} \ar @/_2pc/[uu]_(0.53){\beta \circ \alpha_2}_(0.44){=\beta \circ \alpha_1} \ar @/^2pc/[uu]^{\gamma}
}
\]
\caption{Acyclic category}
\label{acfig}
\end{figure}

\subsection{Lex-shellable acyclic categories}
Many of the definitions and ideas we use from this point forward are simple generalisations of those we introduced when working
with posets. Recall that an edge of a poset $P$ was given by a map from the set of edges of the Hasse diagram of $P$ to another 
poset $\Lambda$. These edges correspond to the cover relations of $P$. For an acyclic categorie $\cat$ we consider 
indecomposable morphisms in the place of cover relations and a {\em labelling} of $\cat$ will be a map from the set of indecomposable
morphisms to a poset $\Lambda$. We assume that our edge labellings satisfy a similar prefix condition as we had for posets.

Rather than using the directed graph with an edge for each morphism, it will be more convenient for us to consider the 
directed graph on $\mathcal{O}(\cat)$ with a directed edge for each indecomposable morphism only.
Note that this does contain enough information to define the category. 

For posets we ordered the maximal chains of a closed interval by the lexicographic order induced by the edge labelling. For acyclic
categories we order the chains of indecomposable morphisms which compose to a given morphism. We give the SBS-condition for acyclic 
categories in Definition \ref{def:sbsac} but first we note that since the categories we work with are augmented, \ie, they have
both initial and terminal objects, the maximal chains of indecomposable morphisms within $\cat$ all compose to the unique morphism 
from the initial object to the terminal object. The lexicographic order then gives total order on the maximal simplices of the nerve
which is a lex-shelling if the SBS-condition is satisfied.

\begin{definition}\label{def:sbsac}
An edge labelling of an augmented acyclic category $\cat$ is a LEX-labelling if it satisfies
the {\em SBS-condition}, given as follows:
For any morphism $m$, and any maximal chain $C=m_1\circ \dots \circ m_k$ in $m$, if $C$ is not lexicographically 
least in $m$ then there exists a subchain $m_i\circ m_{i+1}$ of $C$ composing to $m'$ such that 
$m_i\circ m_{i+1}$ is not lexicographically least in $m'$.
\end{definition}

A finite augmented acyclic category is {\em lex-shellable} if it there exists such a LEX-labelling.

\subsection{Incidence algebras} 
From the directed graph on $\mathcal{O}(\cat)$ with a directed edge for each indecomposable morphism we can construct a 
graded quiver $\Gamma_k[\cat]$ over a field $k$. As before we shall omit $k$ from our notation. 

\begin{remark}
If $\cat$ were actually a poset then, in the notation we use, products in $\Gamma[\cat]$ are in reverse of $\Gamma[P]$. 
This is a result of notational differences for concatenation of paths and composition of morphisms. This change could be 
easily reconciled but we choose to adhere to the conventions.
\end{remark}

The {\em parallel ideal} $I_{\cat}$ is the ideal of $\Gamma[\cat]$ generated by the relations 
$m_1 \circ \dots\circ m_n = m_1' \circ \dots \circ m_r'$ whenever the composition is also equal in $\cat $. 
It is not enough that they share domain and codomain, although of course this is necessary. The {\em incidence algebra} 
is then the quotient $\Gamma[\cat] / I_{\cat}$. This algebra may also be considered as the $k$ vector space with basis given
by the morphisms of $\cat$ and product given by composition of morphisms.

\begin{definition}\label{def:neglexordercat}
Given a labelling of the augmented acyclic category $\cat$ which satisfies the prefix condition, a 
{\em negative degree lexicographic monomial order} for $\Gamma[\cat]$ is any
negative monomial order such that the order on the monomials of equal degree is a refinement of the lexicographic order with 
respect to the labelling.
\end{definition}

From a labelling of $\cat$ we construct a negative degree lexicographic monomial order similar to Definition \ref{def:lexorder}.
We abbreviate `maximal chain of indecomposable morphisms' to `maximal chain'.

\begin{definition}\label{def:lexorderac}
 Given a lex-shelling of an augmented acyclic category $\cat$ 
 we construct a well order $<$ on the monomials of $\Gamma[\cat]$ as follows.
 We label the maximal chains by the total order given by the labelling so that $C_i$ is the $i$th maximal chain.
 To each maximal chain $C_i$ there exists a corresponding `maximal' monomial $m_i$. For any monomial $w$ 
 we define the {\em carrier} of $w$ to be to be earliest chain $C_i$ such that $w$ divides $m_i$.
 We denote the carrier of $w$ by $C_w$.

 For monomials $w$, $v$ we set $w>v$ if:
\begin{itemize}
 \item $\mathrm{deg}\,w < \mathrm{deg}\,v,$ or
 \item $\mathrm{deg}\,w = \mathrm{deg}\,v$ and $l(w) > l(v)$, or
 \item $\mathrm{deg}\,w = \mathrm{deg}\,v$ and $l(w) = l(v)$ and $C_w < C_v$, or
 \item $\mathrm{deg}\,w = \mathrm{deg}\,v$ and $l(w) = l(v)$ and $C_w = C_v$  and $\mathrm{hom}(\dom w,\dom v)\neq \emptyset$.
\end{itemize}
\end{definition}

\subsection{The main result}
\begin{thm}
The parallel ideal $I_{\cat}$ of a finite augmented acyclic category $\cat $ has a quadratic Gr\"obner basis 
with a negative degree lexicographic monomial order if and only if $\cat $ is lex-shellable.
\end{thm}

\begin{proof}
The proof is similar to the proof of Theorem \ref{thm:poset} so we only sketch it here.

($\impliedby$) As in Theorem \ref{thm:poset}, the lex-shelling gives a monomial order for $\Gamma[\cat]_{>0}$. 
Take some element of the parallel ideal for which the initial term $f$ is at least cubic. This corresponds to
a composable chain of indecomposable morphisms. Let $m$ be the composition of this chain. Since $f$ is an initial term of some
element of the ideal the corresponding chain cannot be lexicographically least in $m$, and by the SBS-condition there 
must exists a sub-chain of length two which is not lexicographically least among chains with the same composition. Thus 
$f$ is divisible by some quadratic initial term.   

($\implies$) We assume that we have a negative degree lexicographic monomial order which gives a quadratic Gr\"obner basis
for the parallel ideal. This order arises from a labelling of $\cat$. 
The initial term of any element of the parallel ideal is divisible by a quadratic initial term, hence, similar to Theorem
\ref{thm:poset}, any chain which is not lexicographically least has a subchain of length two which is not lexicographically
in the relevant morphism, showing that this labelling satisfies the SBS-condition 
\end{proof}

\end{document}